\documentclass{amsart}

\usepackage[utf8]{inputenc}
\usepackage{amsmath,amssymb,bbm,enumitem}
\usepackage{hyperref}
\usepackage{amsthm}
\usepackage[capitalize]{cleveref}
\usepackage[square,numbers,sort]{natbib}

\hypersetup{pdfstartview=}

\newcommand\BC{\mathcal C}

\newcommand\BZ{\mathbb Z}
\newcommand\BN{\mathbb N}

\newcommand{\cyc}[1]{\langle #1 \rangle}

\def\BHM#1.#2.#3.#4.{{^{#1}_{#3}\mathcal B^{#2}_{#4}}}

\newcommand\comm\curlyvee
\newcommand\cocomm\curlywedge

\DeclareMathOperator{\End}{End}

\theoremstyle{plain}
\newtheorem{thm}{Theorem}[section]

\newtheorem{lem}[thm]{Lemma}

\theoremstyle{definition}
\newtheorem{df}[thm]{Definition}
\newtheorem{example}[thm]{Example}

\theoremstyle{remark}
\newtheorem{rem}[thm]{Remark}

\crefname{lem}{Lemma}{Lemmas}
\crefname{thm}{Theorem}{Theorems}
\crefname{cor}{Corollary}{Corollaries}
\crefname{prop}{Proposition}{Propositions}
\crefname{example}{example}{examples}
\crefname{df}{Definition}{Definitions}
\crefname{equation}{equation}{equations}

\numberwithin{equation}{section}

\def\clap#1{\hbox to 0pt{\hss#1\hss}}

\def\D{\mathcal D}

\newcommand\inv{^{-1}}

\title{Examples of non-\ensuremath{FSZ} \MakeLowercase{p}-groups for primes greater than three}
\author{Marc Keilberg}
\email{keilberg@usc.edu}
\begin{document}
\begin{abstract}
For any prime $p>3$ and $j\in\BN$ we construct examples of non-$FSZ_{p^j}$ groups of order $p^{p^j+2j-1}$.  In the special case of $j=1$ this yields groups of order $p^{p+1}$, which is the minimum possible order for a non-$FSZ$ $p$-group.
\end{abstract}
\thanks{The author would like to thank the anonymous referee, whose comments helped improved the clarity of this work.}
\maketitle
\section*{Introduction}
The study of the representation categories of semisimple Hopf algebras, and many other more general contexts, have brought forth an interesting invariant of monoidal categories known as (higher) Frobenius-Schur indicators \citep{KSZ2,NS07a,NS07b,NS08,NS10,B1,AG,B2,FGSV,MaN,Nat05,Sc}.  These form generalizations of the classical Frobenius-Schur indicators for a finite group $G$, which for a character $\chi$ of $G$ over $\BC$ and any $m\in\BN$ are defined by
\begin{align}\label{eq:group-ind} \nu_m(\chi) = \frac{1}{|G|}\sum_{g\in G} \chi(g^m).\end{align}

When applied to the Hopf algebra $\D(G)$, the Drinfel'd double of the finite group $G$ over $\BC$, these indicators can be expressed entirely in group theoretical terms. \citet{PS16} has obtained an intriguing description of the FS-indicators of $\D(G)$ in terms of the character tables of centralizers, in particular.  Thus, while the FS-indicators are motivated by Hopf algebraic concerns, they also yield a new and interesting invariant for finite groups. Frobenius-Schur indicators are guaranteed to be algebraic integers in a certain cyclotomic field, and the Galois action on the field also acts on the indicators \citep[Proposition 3.3]{KSZ2}.  In full generality these indicators need not even be real numbers \citep[Example 7.5]{KSZ2}\citep{GMN}, but in the case of $\D(G)$ they are guaranteed to be so\citep[Remark 2.8]{IMM}.  All of the first examples computed in the case of group doubles \citep{K,Co,K2} yielded indicator values in $\BZ$.  Since the higher indicators for $G$ itself are classically known to be integers, this raised the question of whether or not the indicators for $\D(G)$ were always integers for arbitrary $G$.

\citet{IMM} investigated this question, ultimately finding that there were exactly 32 non-isomorphic groups of order $5^6$ with non-integer indicators.  They dubbed the property of having all integer FS-indicators the $FSZ$ property.  They also defined the $FSZ_m$ property, which holds whenever all $m$-th indicators are integers.  For our purposes, \cref{thm:IMM-equiv} below will be taken as our definition of the $FSZ_m$ properties, and therefore the $FSZ$ property.  \citet{IMM} also established that several large families of groups were $FSZ$, including but not limited to the symmetric groups $S_n$; $PSL_2(q)$ for a prime power $q$; and all regular $p$-groups.  On the other hand, the regular wreath product $\BZ_p\wr\BZ_p$ is an irregular $p$-group for all primes $p$, and this was shown to be $FSZ$ \citep[Example 4.4]{IMM}, thereby establishing that the class of $FSZ$ $p$-groups properly contains the class of regular $p$-groups.  It is interesting to ask what can be said about the properties of irregular non-$FSZ$ $p$-groups, or alternatively of irregular $FSZ$ $p$-groups.

It is the goal of this note to exhibit an infinite family of non-$FSZ$ $p$-groups for arbitrary primes $p>3$.   The construction, in particular, establishes that there are always non-$FSZ$ $p$-groups of order $p^{p+1}$ when $p>3$, which is well-known to be the minimum order possible for an irregular $p$-group.

We will take $\BN=\{1,2,...\}$ to be the set of positive integers.

\section{The construction}\label{sec:main}
Fix an odd prime $p$ and an integer $j\in\BN$.

Consider the abelian $p$-group \[P_{p,j}=\BZ_{p^{j+1}}\times \BZ_{p}^{p^{j}-2},\] with generators $a_1,...,a_{p^{j}-1}$ where $a_1$ has order $p^{j+1}$ and the rest have order $p$.  We define an endomorphism $b_{p,j}$ of $P_{p,j}$ by
\begin{align*}
  a_1 \mapsto&\ a_1 a_2\inv\\
  a_k\mapsto&\ a_k a_{k+1}, \ 1<k<p^{j}-1\\
  a_{p^{j}-1}\mapsto&\ a_{p^{j}-1} a_1^{-p^{j}}.
\end{align*}
It is convenient to write $b_{p,j}$ as a matrix $B_{p,j}$ which acts on the left in the obvious fashion, and whose first row of entries can be taken modulo $p^{j+1}$ and the remaining entries may be taken modulo $p$.  We have
\[B_{p,j}=\begin{pmatrix}
  1 & 0 & 0 & \cdots & 0 & 0 & -p^j\\
  -1&1&0&\cdots&0&0&0\\
  0&1&1&\cdots&0&0&0\\
  \vdots&&&&&&\vdots\\
  0&0&0&\cdots&1&1&0\\
  0&0&0&\cdots&0&1&1
\end{pmatrix}.\]
The entries of $B_{p,j}^k$ for $1\leq k\leq p^{j}-2$ are then naturally described by the values
\[T_{i,k} = {i\choose k}\] of Pascal's Triangle.  For example
\[B_{p,j}^2=\begin{pmatrix}
  1 & 0 & 0 & \cdots & 0 & -p^j & -2p^j\\
  -2&1&0&\cdots&0&0&0\\
  -1&2&1&\cdots&0&0&0\\
  \vdots&&&&&&\vdots\\
  0&0&0&\cdots&2&1&0\\
  0&0&0&\cdots&1&2&1
\end{pmatrix}\]
and $B_{p,j}^{p^{j}-2}$ is given by
\[\begin{pmatrix}
  1 & -p^j &  \cdots & -T_{p^{j}-2,p^{j}-4}p^j & -T_{p^{j}-2,p^{j}-3}p^j\\
  -T_{p^{j}-2,1}&1&\cdots&0&0\\
  -T_{p^{j}-2,2}&T_{p^{j}-2,1}&\cdots&0&0\\
  \vdots&\vdots&\vdots&\vdots&\vdots\\
  -T_{p^{j}-2,p^{j}-3}&T_{p^{j}-2,p^{j}-3}&\cdots&1&0\\
  -1&T_{p^{j}-2,p^{j}-3}&\cdots&T_{p^{j}-2,1}&1
\end{pmatrix}.\]
Indeed, the entries of $B_{p,j}^k$ are determined by Pascal's triangle for arbitrary $k$, just that for $k> p^{j}-2$ we can no longer fit entire rows of the triangle in the rows or columns.  Nevertheless the pattern is straightforward. We remind the reader that $T_{p^t,k}={p^t\choose k}$ is divisible by $p$ for all $t\in\BN$ and $0<k<p^t$, and is equal to $1$ for $k=0$ and $k=p^t$.  This elementary property is essential to several of the calculations we will do, as it regularly ensures that many entries are zero, and will be used without further mention.

\begin{lem}\label{lem:b-is-aut}
The endomorphism $b_{p,j}\in\End(P_{p,j})$ is an automorphism of order $p^j$.
\end{lem}
\begin{proof}
  The formula for $B_{p,j}^{p^j-2}$ above shows that the upper left entry for $B^{p^j-1}=B\cdot B^{p^j-2}$ is $p^j+1$, which is not congruent to $1$ modulo $p^{j+1}$.  Since the powers $B^k$ for $1\leq k < p^j-1$ have a $-1$ entry in the first column, it follows that $B^{k}$ is not the identity matrix for $1\leq k\leq p^j-1$.  To finish showing that $B$ has order $p^j$, we write
  \[ B = I + S,\]
  where $I$ is the identity matrix and $S$ is the matrix
  \begin{align}\label{eq:circulant} S = \begin{pmatrix}
    0 & 0 &\cdots& 0 & -p^j\\
    -1 & 0 &\cdots & 0 & 0\\
    0 & 1 & \cdots & 0 & 0\\
    \vdots&\vdots&\vdots&\vdots&\vdots\\
    0&0&\cdots&1&0
  \end{pmatrix}.\end{align}
  By the binomial formula,
  \[ B^{p^j} = I + S^{p^j} + \sum_{k=1}^{p^{j-1}} {p^j\choose k} S^k.\]
  The binomial coefficients in the summation are all divisible by $p$, whence only the first row of the powers $S^k$ can possibly contribute a non-zero term to the summation.  $S$ itself behaves very much like a circulant matrix, and in particular can have its arbitrary powers computed easily: for each successive power, shift the entries of the previous power to the left one column, set the last column to all zeros, and multiply the first column by -1.  For example,
  \[ S^2 = \begin{pmatrix}
    0&0&\cdots&0&-p^j&0\\
    0&0&\cdots&0&0&0\\
    -1&0&\cdots&0&0&0\\
    \vdots&\vdots&\vdots&\vdots&\vdots&\vdots\\
    0&0&\cdots&1&0&0
  \end{pmatrix}.\]
  We can then see that, for $1\leq k<p^j$, the first row of $S^k$ has a single non-zero entry which is equal to $-p^j$ for $k<p^j-1$, and is equal to $p^j$ for $k=p^j-1$.  Since this is multiplied by the binomial coefficient and taken modulo $p^{j+1}$, we see that every term in the summation vanishes.  Lastly we see that $S^{p^j}$ is the zero matrix.  Thus, $B^{p^j}$ is the identity matrix, as desired.
\end{proof}

We can now define the family of groups for which we wish to study the $FSZ$ properties.

\begin{df}
  Let $p$ be an odd prime and $j\in\BN$.  Define \[S(p,j) = P_{p,j}\rtimes \cyc{b_{p,j}}.\]  This is a group of order $p^{p^{j}+2j-1}$.
\end{df}
By considering the eigenvectors of $B$ for the eigenvalue $1$, we see that $S(p,j)$ has center $\cyc{a_1^p}\cong\BZ_{p^j}$.  We identify $P_{p,j}$ and $\cyc{b_{p,j}}$ as subgroups of $S(p,j)$ in the usual fashion, and for simplicity we denote $b_{p,j}$ by simply $b$ whenever convenient, and similarly for $B_{p,j}$.

The group $S(p,j)$ is defined in such a way as to make computing $p^j$-th powers relatively easy, after a bit of initial work.  To help us investigate how to take $p^j$-th powers in $S(p,j)$, we introduce the following.
\begin{df}\label{def:y-def}
For $0\leq k < j$ we define the matrices
\[Y_{p,j}(p^k) = \sum_{m=0}^{p^{j-k}-1} B^{m p^k}.\]
\end{df}
These can be viewed as endomorphisms of $P_{p,j}$ by acting on the left, in the same fashion that $B$ acts.

\begin{lem}
  We have a block decomposition with a $1\times 1$ entry in the upper left corner
  \begin{align}\label{eq:Y1mat}
    Y_{p,j}(1) =  \begin{pmatrix} 2p^j & 0\\ 0&0\end{pmatrix}.
  \end{align}
\end{lem}
\begin{proof}
We have the identity $B Y_{p,j}(1) = Y_{p,j}(1)$.  Thus the columns of $Y_{p,j}(1)$ are all eigenvectors of $B$ with eigenvalue 1, and it is easily checked that all such eigenvectors have zeroes in every entry except (possibly) the first one, which must be divisible by $p$.  It follows that we have a block decomposition with a $1\times 1$ entry in the upper left corner given by
\[ Y_{p,j} =  \begin{pmatrix} c*p & v\\ 0&0\end{pmatrix}\]
for some integer $c$ and some (row) vector of integers $v$.  Indeed, every entry of $v$ must be divisible by $p$.    As noted in the proof of \cref{lem:b-is-aut}, it is easily seen that the $(1,1)$ entry of $B_{p,j}^{p^j-1}=B*B^{p^j-2}$ is exactly $p^j+1$, from which it follows that the $(1,1)$ entry of $Y_{p,j}(1)$ is $2p^j$.  It remains to show that $v$ is the zero vector (modulo $p^{j+1}$).

To this end we note that we also have $Y_{p,j}(1)B = Y_{p,j}(1)$, or equivalently that $Y_{p,j}(1)(B-I)=0$, where $I$ is the identity matrix.  As in \cref{lem:b-is-aut}, we write $S=B-I$, which takes the form given in \eqref{eq:circulant}.  Thus the row vectors of $Y_{p,j}$ are annihilated by $S$ when acted on from the right.  Writing $v=(v_1,...,v_{p^{j}-2})$, we see that $(2p^j,v)S = (-v_1,v_2,...,v_{p^j-2},0) = 0 \mod p^{j+1}$.  This shows that $v$ is the zero vector, as desired, and so completes the proof.
\end{proof}

We also have the following relations for the remaining $Y_{p,j}(p^k)$.
\begin{lem}
For all $0<k<j$ we have a block decomposition with a $1\times 1$ entry in the upper left corner
\begin{align}\label{eq:Ypmats}
  p^k Y_{p,j}(p^k) =& \begin{pmatrix}p^j&0\\0&0\end{pmatrix}, \ 1\leq k < j.
\end{align}
\end{lem}
\begin{proof}
  Since all rows but the first are taken modulo $p$, the scalar multiplication by $p^k$ automatically forces all entries of $Y_{p,j}(p^k)$ other than possibly those on the first row to be zero.  In particular, we have established the bottom two rows of the desired decomposition.  Indeed, the scalar multiplication by $p^k$ means we need only consider the elements in the first row to calculate $p^k Y_{p,j}(p^k)$.

  We now consider the entries in the first row. We easily see that, by assumptions on $k$, the $(1,1)$ entry of every matrix in the sum defining $Y_{p,j}(p^k)$ is $1$, and there are $p^{j-k}$ matrices in the summation, whence the $(1,1)$ entry of $p^k Y_{p,j}(p^k)$ is $p^j$.  Recall that the first row is always taken modulo $p^{j+1}$.  In the first row of $B^{p^k}$ the only nonzero entries are the first and last ones, which are 1 and $-p^j$ respectively.  Indeed, all entries on the diagonal of $B^{p^k}$ are $1$, and the only non-zero entries below the diagonal are $\pm 1$ modulo $p$.  If follows that every entry in the first row of any power of $B^{p^k}$, other than the first entry, is divisible by $p^j$.  Subsequently, every entry in the first row of $p^k Y_{p,j}(p^k)$, except the first one, is divisible by $p^{j+k}\equiv 0 \mod p^{j+1}$ since $k\geq 1$.  This completes the proof.
\end{proof}

We can now state how these matrices are used to describe arbitrary $p^j$-th powers in $S(p,j)$.  Namely, fixing $q\in P_{p,j}$ and $b^k\in\cyc{b}$ with $|b^k|=p^{j-t}$, then
\begin{align}\label{eq:powers}
  (qb^k)^{p^j} =&\ p^{t}Y_{p,j}(p^t)q.
\end{align}
The fact that $P_{p,j}$ is abelian is essential to this formula.  In particular, it is needed to be sure that the value depends only on the order of $b^k$.  The reader may find it worthwhile to see how this holds in the particular case of $k\equiv -1\mod p^{j+1}$, as we have previously noted that this is the only power of $B$ for which the $(1,1)$ entry is not congruent to $1$.  As a result of this identity, all $p^j$-th powers in $S(p,j)$ yield the subgroup $\cyc{a^{p^j}}\subseteq Z(S(p,j))$.

\begin{rem}
The matrices $p^k Y_{p,j}(p^k)$, in particular $Y_{p,j}(1)$, are higher dimensional analogues of the integer parameter $d$ appearing in \citep{K}.  The parameter $d$ controlled the existence of negative indicators in the double of the groups under consideration in \citep{K}, in much the same way that $Y_{p,j}$ will dictate the existence of non-integer indicators here.  More generally, such objects naturally arise when considering the $FSZ$ property for groups of the form $A\rtimes C$ where $A$ is abelian and $C$ is cyclic, such as in \citep[Example 4.4]{IMM}.
\end{rem}

Now that we understand how to take $p^j$-th powers in $S(p,j)$, we can begin investigating what \citet{IMM} called the $FSZ_{p^j}$ property of $S(p,j)$.  We recall the following.
\begin{df}
  For any group $G$, $n\in\BN$, and $g,u\in G$, define
  \[ G_n(u,g) = \{ a\in G \ : \ a^n = (au\inv)^n =g \}.\]
\end{df}
Of necessity, $G_n(u,g)\neq \emptyset$ implies $[u,g]=1$, and indeed for fixed $g$ they are subsets of $C_G(g)$.  These sets characterize the $FSZ_n$ properties, as shown by the following.
\begin{thm}\citep[Corollary 3.2]{IMM}\label{thm:IMM-equiv}
    Let $n\in\BN$ and $G$ a finite group.  Then $G$ is an $FSZ_n$-group if and only if for all commuting pairs of elements $u,g$ and all integers $m$ coprime to $|G|$ we have
    \[ |G_n(u,g)|=|G_n(u,g^m)|.\]
\end{thm}
We can now state and prove the main result of the paper.

\begin{thm}\label{thm:Gsets}
  Let notation be as above and set $G=S(p,j)$ for any odd prime $p>3$ and $j\in\BN$.  Then $G_{p^j}(ba_1,a_1^{p^j})=\emptyset$ and $G_{p^j}(ba_1,a_1^{2p^j})\neq\emptyset$.

  In particular, $S(p,j)$ is non-$FSZ_{p^j}$.
\end{thm}
\begin{proof}
  We first note that the assumption $p>3$ is necessary, since when $p=3$ we have $a_1^{2p^j}=a_1^{-p^j}$ and by \citep[Lemma 2.7]{IMM} we always have a bijection $G_n(u,g)\to G_n(u,g\inv)$ for any $G$, $n\in\BN$, and $u,g\in G$.

  Fix $u=ba_1$ and set $Y=\cyc{a_2,...,a_{p^j-1}}$ for the remainder of the proof.  Every element $a\in G$ can be uniquely written in the form $a=a_1^{j_1}y b^{-k}\in G$ for some $y\in Y$.  By \cref{eq:powers,eq:Y1mat,eq:Ypmats} the value of $a^{p^j}$ does not depend on $y$, so we may suppress elements of $Y$ for the rest of the proof.  It follows that when determining the membership of $G_{p^j}(ba_1,g)$ we will naturally break things down into cases, depending on the orders of $b^k$ and $b^{k+1}$.

  First consider the case that $|b^k|=1$.  Then \[a^{p^j}=a_1^{j_1 p^j}=g,\] while \[(au\inv)^{p^j} = Y_{p,j}(1)(a_1^{j_1-1}) = a_1^{2(j_1-1)p^j}=g.\]  These equalities are consistent if and only if $j_1\equiv 2\bmod p$.  In particular, we have no contribution from elements of this form when $g=a_1^{p^j}$, but do have contributions from such elements when $g=a_1^{2p^j}$.

  Now suppose $|b^k|=|b^{k+1}|=p^j$.  Then \[a^{p^j} = Y_{p,j}(1)a_1^{j_1} = a_1^{2j_1 p^j}\] and \[(au\inv)^{p^j} = Y_{p,j}(1)a_1^{j_1-1} = a_1^{2(j_1-1)p^j}.\]  We point out to the reader that this identity holds even when $k=1$, as in this case putting $au\inv$ into the desired form requires we apply $B^{-1}$ to $a_1$, and we have previously noted that this is the unique power of $B$ whose (1,1) is $p^j+1$ instead of 1.  Regardless, these values can never be equal, so we have no contributions from elements of this form to the sets $G_{p^j}(u,g)$ for any choice of $g$.

  Next suppose $k\equiv -1\mod p^j$.  Then \[a^{p^j} = Y_{p,j}(1)a_1^{j_1} = a_1^{2j_1 p^j}\] while \[(au\inv)^{p^j} = a_1^{(j_1-1) p^j}.\]  These are equal if and only if $j\equiv -1\mod p$.  Since $p>3$, we conclude that for  $g\in\{a_1^{p^j},a_1^{2p^j}\}$ there are no contributions from elements of this form.

  For the next case, suppose $|b^k|=p^{j-t}$ for some $0<t<j$, which implies $|b^{k+1}|=p^j$.  It follows that \[a^{p^j}=a_1^{j_1 p^j}\] and \[(au\inv)^{p^j} = a_1^{2(j_1-1)p^j}.\]  These are equal if and only if $j_1\equiv 2\mod p$.  Therefore for $g=a_1^{p^j}$ there are no contributions from elements of this form.  But for $g=a_1^{2p_j}\}$ contributions from elements of this form do exist.

  Finally, suppose $|b^{k+1}|=p^{j-t}$ for some $0<t<j$, which implies $|b^k|=p^j$.  Then we have
  \[ a^{p^j} = a_1^{2p^j j_1}\]
  while
  \[ (au\inv)^{p^j} = a_1^{p^j (j_1-1)}.\]
  These values are equal if and only if $j_1\equiv -1\mod p$, so again since $p>3$ we conclude that for $g\in\{a_1^{p^j},a_1^{2p^j}\}$ there are no contributions from elements of this form.

  This completes the proof except for the final claim, which follows immediately from \cref{thm:IMM-equiv}.
\end{proof}
\begin{example}
  For $p>3$ we have that $S(p,1)$ is a group of order $p^{p+1}$ that is not $FSZ_p$, and this is the minimum possible order for any non-$FSZ$ $p$-group.  Indeed, $S(5,1)$ is SmallGroup($5^6$,632) in \citet{GAP4}, which is the smallest id number amongst the 32 non-$FSZ$ groups of order $5^6$ found by \citet{IMM}.

  For $p>3$ and $j>1$ we do not know if $S(p,j)$ has minimal order amongst the non-$FSZ_{p^j}$ $p$-groups.
\end{example}
\begin{example}
  \citet{IMM} used \citet{GAP4} to verify that there are no non-$FSZ$ 2-groups of order at most $2^9$.  The author has verified, with the help of the GAP functions in \citep{IMM,PS16}, that there are no non-$FSZ$ 3-groups of order at most $3^7$.  It remains an open question if non-$FSZ$ 2-groups or 3-groups exist, and if they do what their minimum orders are.  The constructions here, and several attempts at modifications thereof, run into the usual issues for the primes $2,3$.
\end{example}

\bibliographystyle{plainnat}
\bibliography{../references}
\end{document}